\documentclass[12pt]{amsart}
\usepackage{amscd, amssymb, pgf, tikz, hyperref}

\newcommand{\field}[1]{\mathbb{#1}}
\newcommand{\CC}{\field{C}}

\newcommand{\MM}{\field{M}}
\newcommand{\TT}{\field{T}}
\newcommand{\ZZ}{\field{Z}}
\newcommand{\RR}{\field{R}}
\newcommand{\Aa}{\mathcal{A}}

\newcommand{\Cc}{\mathcal{C}}
\newcommand{\Ee}{\mathcal{E}}

\newcommand{\Hh}{\mathcal{H}}
\newcommand{\Kk}{\mathcal{K}}
\newcommand{\Ll}{\mathcal{L}}
\newcommand{\Mm}{\mathcal{M}}
\newcommand{\Oo}{\mathcal{O}}
\newcommand{\Zz}{\mathcal{Z}}

\newcommand{\la}{\langle}
\newcommand{\ra}{\rangle}
\newcommand{\ve}{\varepsilon}

\newtheorem{thm}{Theorem}[section]
\newtheorem{cor}[thm]{Corollary}

\newtheorem{prop}[thm]{Proposition}

\theoremstyle{definition}
\newtheorem{dfn}[thm]{Definition}

\theoremstyle{remark}
\newtheorem{rmk}[thm]{Remark}

\newtheorem{example}[thm]{Example}
\newtheorem{examples}[thm]{Examples}
\newtheorem*{examples*}{Examples}

\numberwithin{equation}{subsection}

\title{Symmetries of the $C^*$-algebra of a vector bundle}

\author{Valentin Deaconu}
\address{Valentin Deaconu \\ Department of Mathematics \& Statistics\\ University
of Nevada\\ Reno NV 89557-0084\\ USA} \email{vdeaconu@unr.edu}


\begin{document}
\begin{abstract}We consider $C^*$-algebras constructed from compact group actions  on complex vector bundles $E\to X$ endowed with a Hermitian metric. An action of $G$   by isometries on $E\to X$ induces an  action  on the $C^*$-correspondence $\Gamma(E)$  over $C(X)$ consisting of continuous sections, and on the associated Cuntz-Pimsner algebra $\Oo_E$, so we can study the crossed product $\Oo_E\rtimes G$.
If the action  is free and rank $E=n$, then we prove that $\Oo_E\rtimes G$ is Morita-Rieffel equivalent to a field of Cuntz algebras $\Oo_n$ over the orbit space $X/G$.
 If the action   is fiberwise, then $\Oo_E\rtimes G$ becomes a continuous field of crossed products $\Oo_n\rtimes G$. For transitive  actions, we show that $\Oo_E\rtimes G$ is Morita-Rieffel equivalent to a graph $C^*$-algebra.
\end{abstract}
\maketitle
\section{introduction}

\bigskip

Given a  vector bundle $p:E\to X$ over a compact Hausdorff space $X$, we denote by Aut$(E)$ the  group  consisting of pairs $(\varphi, \Phi)$ where $\varphi :X\to X$ is a homeomorphism, $\Phi:E\to E$ is a map such that $p\circ \Phi=\varphi\circ p$ and if $E_x=p^{-1}(x)$ is the fiber over $x$, then all restrictions
$\Phi_x:E_x\to E_{\varphi(x)}$ are linear isomorphisms. An  action of a group $G$ on $E$ is understood as a homomorphism $G\to$Aut$(E)$. If $G$ acts trivially on $X$,   then we say that the action on $E$ is fiberwise. The set of fiberwise automorphisms is a subgroup  of Aut$(E)$, denoted Aut$_X(E)$.

In this paper, we consider $C^*$-algebras constructed from compact group actions  on complex vector bundles $E\to X$ endowed with a Hermitian metric. By taking continuous sections, we get a finitely generated projective Hilbert module $\Gamma(E)$ over $A=C(X)$,  viewed as a   $C^*$-correspondence   with the same left and right multiplications and the usual inner product.  The associated Cuntz-Pimsner algebra $\Oo_{C(X)}(\Gamma(E))$, also denoted $\Oo_E$ for short,  is a continuous field of Cuntz algebras, see Proposition 2 in \cite{V}. An action of $G$  on $E$ by isometries induces an  action of $G$ on $\Gamma(E)$ and on $\Oo_E$, so  one can study the crossed product $\Oo_E\rtimes G$. From the theorem of Hao-Ng \cite{HN}, we have the isomorphism
$\Oo_E\rtimes G\cong \Oo_{C(X)\rtimes G}(\Gamma(E)\rtimes G)$, so it is useful to understand the crossed product correspondence $\Gamma(E)\rtimes G$ as a kind of noncommutative bundle over $C(X)\rtimes G$, which in some cases is Morita-Rieffel equivalent to an abelian $C^*$-algebra.

In particular, if the action of $G$ is free and rank $E=n$, then we show that $\Oo_E\rtimes G$ is Morita-Rieffel equivalent to a field of Cuntz algebras $\Oo_n$ over the orbit space $X/G$.
 If the action of $G$  is fiberwise, then we prove that $\Oo_E\rtimes G$ becomes a continuous field of crossed products $\Oo_n\rtimes G$. For transitive  actions, the vector bundle $E$ has a special form, and $\Oo_E\rtimes G$ is Morita-Rieffel equivalent to a graph $C^*$-algebra.

 \bigskip

\section{Group actions on vector bundles}

We collect here some facts about $G$-vector bundles over $G$-spaces. For simplicity, in section 1.6 of \cite{A}, the author considers only  the finite group case, but assuming continuity of the maps, all the results generalize to compact groups, see \cite{S}. 

Let's fix a compact group $G$ with identity $e$. A $G$-space is a topological space $X$ together with a continuous action $G\times X\to X$ written $(g,x)\mapsto g\cdot x$ satisfying the usual conditions $g\cdot(g'\cdot x)=(gg')\cdot x$ and $e\cdot x=x$.

\begin{dfn} 
If $X$ is a $G$-space, then a $G$-vector bundle on $X$  is a (locally trivial) complex vector bundle $p: E\to X$ such that $E$ is a $G$-space, $p$ is a $G$-map i.e. $p(g\cdot v)=g\cdot p(v)$, and the  maps $E_x\to E_{g\cdot x}$ between fibers are linear.

 \end{dfn}

\begin{example}

If $X$ is a differentiable manifold and $G$ is a group acting smoothly on $X$, then the complexified tangent bundle $E=TX\otimes \CC$ and all the associated tensor bundles $E^{\otimes n}$ become  $G$-vector bundles.

\end{example}

\begin{example}
If $E$ is a complex vector bundle on $X$, then the $k$-fold tensor product $E^{\otimes k}$ becomes an $S_k$-vector bundle on $X$, where the symmetric group $S_k$ fixes the points in $X$ and permutes the tensor factors.
\end{example}

\begin{rmk} If $X$ reduces to a point, then a $G$-vector bundle is just a finite dimensional representation of $G$. If $X$ is a trivial $G$-space, i.e. $g\cdot x=x$ for all $g\in G$ and $x\in X$, then a $G$-vector bundle is a family of representations  of $G$ varying continuously with $x$. 
\end{rmk}

Recall that $X$ is a free $G$-space if $g\neq e$ implies $g\cdot x\neq x$. We denote by $X/G$ the orbit space, and let $\pi:X\to X/G$ be the canonical map. 

\begin{rmk}
Given a $G$-vector bundle $E$ over a free $G$-space $X$, then $E$ is necessarily a free $G$-space, and  $E/G$ has a natural vector bundle structure over $X/G$. In fact, $E/G\to X/G$ is locally isomorphic to $E\to X$ and the action of $G$ on $E/G$ is trivial. 

If $X$ is a free $G$-space and $p:F\to X/G$ is any vector bundle over $X/G$, then the pull back \[\pi^*F=\{(x,v)\in X\times F\;\mid \; \pi(x)=p(v)\}\] is a $G$-vector bundle over $X$ with action $g\cdot (x,v)=(g\cdot x, v)$.
\end{rmk}
 In particular, we have
\begin{prop}
$G$-vector bundles over a free $G$-space $X$ correspond bijectively to vector bundles over $X/G$.
\end{prop}
Assume now that the action on $X$ is trivial, hence $G$ acts on $E\to X$ fiberwise. By averaging over $G$, we get a projection \[P:E\to E^G,\; P(v)=\int_Gg\cdot v\; dg,\] where $E^G$ is the subbundle of fixed vectors and $dg$ is the normalized Haar measure on $G$. In particular, if $E,F$ are $G$-bundles over $X$ and the action on $X$ is trivial, we have \[Hom_G(E,F)\cong Hom(E,F)^G,\] where  $Hom(E,F)$ is the vector bundle over $X$ with fibers $Hom(E_x, F_x)$ and $Hom_G(E,F)$ is the set of $G$-equivariant morphisms.

Denote by $\{V_i\}_{i\in I}$ the set of irreducible representations of $G$, so that any representation $V$ of $G$ is isomorphic to a direct sum $\bigoplus_{i\in I} n_iV_i$, where $n_i$ are multiplicities.  Take $W_i=X\times V_i$ to be the product bundle with the natural $G$-action $g\cdot (x,v)=(x,g\cdot v)$.  If $E$ is a $G$-bundle over $X$,  then the bundle map \[\bigoplus_i W_i\otimes Hom_G(W_i,E)\to E\] is an isomorphism  of $G$-bundles and the action on $Hom_G(W_i,E)$ is trivial. We conclude that 
\begin{prop} 
For $X$ a trivial $G$-space, every $G$-bundle $E$ over $X$ is isomorphic to a direct sum $\bigoplus_i W_i\otimes E_i$, where $W_i=X\times V_i$ as above and $E_i=Hom_G(W_i,E)$ are vector bundles with trivial $G$-action. 
\end{prop}
Suppose now that  $X$ is a transitive $G$-space and fix $x\in X$ with stabilizer group $G_x$. Then there is a natural continuous bijection of $G$-spaces \[f:G/G_x\to X, f(tG_x)=t\cdot x.\]
If $H\subseteq G$ is a subgroup and $p:E\to G/H$ is a $G$-vector bundle over the homogeneous space $G/H=X$, then the stabilizer group $H$ of the coset $x=eH$ acts on the fiber $E_x$ as a group representation and $E_x$ becomes an $H$-module. Let $G\times_HE_x$ denote the orbit space of $G\times E_x$ under the action of $H$ given by $h(g,v)=(gh^{-1}, hv)$. Then $G\times_HE_x$ becomes a $G$-vector bundle over $G/H$ via the map $q: G\times_H E_x\to G/H$, $q(g,v)=gH$ and the action $g'(g,v)=(g'g,v)$. Moreover, the map \[\Phi: G\times_H E_x\to E,\; \Phi(g,v)=gv\] is an isomorphism of $G$-bundles over $G/H$.

Conversely, to each representation $V$ of $H$, we can associate a vector bundle $G\times_H V$ over $G/H$. Hence we obtain

\begin{prop} Any $G$-vector bundle over the homogeneous space $G/H$ is of the form $G\times_H V$ for some $H$-module $V$.
\end{prop}
\bigskip


\section{Group actions on $C^*$-correspondences and crossed products}
\bigskip

Recall that given  $C^*$-algebras $A$ and $B$, a $C^*$-correspondence $\Hh={}_A\Hh_B$ from $A$ to $B$ is a right Hilbert $B$-module with a $*$-homomorphism $\phi:A\to\Ll(\Hh)$ which gives the left multiplication. Here $\Ll(\Hh)=\Ll_B(\Hh)$ denotes the set of adjointable operators, which are automatically bounded and $B$-linear. We often write $a\xi$ for $\phi(a)\xi$. 

If the map $\phi$ is injective, then $\Hh$ is called faithful. If the span of $\phi(A)\Hh$ is dense in $\Hh$, then $\Hh$ is called essential or nondegenerate. If the inner products $\la \xi, \eta\ra$ generate $B$, then $\Hh$ is called full.

\begin{dfn}
A representation  of a $C^*$-correspondence ${}_A\Hh_B$ is a triple $(\pi_A, \tau_\Hh, \pi_B)$ consisting of nondegenerate representations $\pi_A$ and $\pi_B$ of $A$ and $B$ on Hilbert spaces $H_A$ and $H_B$ respectively, together with a linear map $\tau_\Hh:\Hh\to \Ll(H_B, H_A)$ such that
\[\tau_\Hh(a\xi)=\pi_A(a)\tau_\Hh(\xi),\; \tau_\Hh(\xi b)=\tau_\Hh(\xi)\pi_B(b),\; \pi_B(\la \xi, \eta\ra)=\tau_\Hh(\xi)^*\tau_\Hh(\eta)\]
for all $a\in A, \xi, \eta\in \Hh$ and $b\in B$. Here $\Ll(H_B, H_A)$ is the set of linear bounded operators between the Hilbert spaces $H_B$ and $H_A$.
\end{dfn}

Given a locally compact group $G$, an action on ${}_A\Hh_B$ is determined by a homomorphism $\rho:G\to {\mathcal L}_{\mathbb C}({\mathcal H})$ such that $\rho_g$ is a ${\mathbb C}$-linear isomorphism and $g\mapsto \rho_g\xi$ is continuous for all $\xi\in \Hh$,  and by continuous actions of $G$ on $A$ and $B$ given by $\alpha:G\to$Aut$(A)$, $\beta:G\to$Aut$(B)$ with compatibility relations
\[ \langle \rho_g(\xi),\rho_g(\eta)\rangle=\beta_g(\langle\xi,\eta\rangle) ,\;\rho_g(\xi b)=\rho_g(\xi)\beta_g (b),\; \rho_g(a\xi)=\alpha_g( a)\rho_g(\xi),\]
where $\xi, \eta \in {\mathcal H}, a\in A, b\in B$.
 
 \bigskip

\begin{dfn}Suppose the locally compact group $G$ acts on the $C^*$-correspondence ${\mathcal H}$ from $A$ to $B$ via maps $\rho, \alpha, \beta$ as above.
 The  crossed product  ${\mathcal H}\rtimes_\rho G$ 
  is a  $C^*$-correspondence from $A\rtimes_\alpha G$ to $B\rtimes_\beta G$, and it can be obtained as the completion of $C_c(G,{\mathcal H})$ using the operations
  \[\langle \xi, \eta\rangle(t)=\int_G\beta_{s^{-1}}( \langle \xi(s),\eta(st)\rangle) ds,\]
                \[(\xi\cdot f)(t)=\int_G\xi(s)\beta_s(f(s^{-1}t))ds,\]
                \[(h\cdot\xi)(t)=\int_Gh(s)\rho_s(\xi(s^{-1}t))ds,\]
  where $\xi,\eta\in C_c(G,{\mathcal H}), f\in C_c(G,B), h\in C_c(G,A)$. 
 \end{dfn} 
  \begin{rmk}The crossed product correspondence ${\mathcal H}\rtimes_\rho G$ can also be obtained as ${\mathcal H}\otimes_{\varphi}(B\rtimes_\beta G)$, where $\varphi: B\to {\mathcal L}(B\rtimes_\beta G)$ is the embedding of $B$ in the multiplier algebra of $B\rtimes_\beta G$, regarded as a Hilbert module over itself. 
  
  The group $G$ acts on ${\mathcal H}\otimes_{\varphi}(B\rtimes_\beta G)$ by
  $g(\xi\otimes f)=\rho_g(\xi)\otimes  u_gfu_g^{-1}$, where $f\in B\rtimes_\beta G$ and $u:G\to \Ll(B\rtimes_\beta G)$ is the unitary representation such that $\beta_g(b)=u_gbu_g^{-1}$. If $U_g(\xi\otimes f)=\rho_g(\xi)\otimes u_gf$, then $U_g\in \Ll(\Hh\rtimes_\rho G)$ with $U_g^*=U_{g^{-1}}$ and $U_g(\xi\otimes f)U_{g^{-1}}=\rho_g(\xi)\otimes u_gfu_g^{-1}$.
  
  The left  and right multiplications on ${\mathcal H}\otimes_{\varphi}(B\rtimes_\beta G)$ are given by convolution,
  \[(h\cdot(\xi\otimes f))(t)=\int_Gh(s)\rho_s(\xi)\beta_s(f(s^{-1}t))ds,\;\;(\xi\otimes f)\cdot f'=\xi\otimes ff'\]
  for $h\in C_c(G,A)$ and $ f,f'\in C_c(G,B)$. The inner product formula could be expressed as
  \[\langle\xi\otimes f, \eta\otimes f'\rangle=f^*\langle \xi, \eta\rangle f',\]
  where  $\xi, \eta\in {\mathcal H}$ and $f^*(t)=\Delta(t)^{-1}\alpha_t(f(t^{-1})^*)$, with $\Delta$  the modular function on $G$.
  The isomorphism with the crossed product defined using a completion of $C_c(G,{\mathcal H})$ is induced by the function 
  \[\Phi:\Hh\otimes C_c(G,B)\to C_c(G,\Hh), \; \Phi(\xi\otimes f)(g)=\xi f(g).\]
\end{rmk}

\bigskip

\begin{dfn} Suppose the locally compact group $G$ acts on the $C^*$-correspondence ${\mathcal H}$ from $A$ to $B$ via maps $\rho, \alpha, \beta$ as above.
If $(\pi_A, u,H_A)$ and $(\pi_B,v,H_B)$ are covariant representations of $(A,G, \alpha)$ and $(B,G,\beta)$, then $(\pi_A, \tau_{\Hh}, \pi_B,u,v)$ is called  a covariant representation of $(\Hh,G,\rho)$  in $\Ll(H_B, H_A)$ if
\[\tau_{\Hh}(\rho_g(\xi))=u_g\tau_{\Hh}(\xi)v_g^*.\]
\end{dfn}
In this case, we can define the representation $\tau_{\Hh}\times v$ of $\Hh\rtimes_\rho G$ into $\Ll(H_B,H_A)$ by
\[(\tau_{\Hh}\times v)(\xi)=\int_G\tau_{\Hh}(\xi(s))v_sds\]
for $\xi\in C_c(G,\Hh)$. 
Then $(\pi_A\times u,\tau_{\Hh}\times v, \pi_B\times v)$ becomes a representation of the $C^*$-correspondence $\Hh\rtimes_\rho G$ from $A\rtimes_\alpha G$ to $B\rtimes_\beta G$.

\begin{example} (The reduced crossed product) Consider $(\pi_A,\tau_\Hh, \pi_B)$  a representation of  ${}_A\Hh_B$ such that $\pi_A:A\to \Ll(H_A)$ and $\pi_B:B\to \Ll(H_B)$  are faithful. If $\lambda$ is the left regular representation of $G$ given by $(\lambda_s f)(t)=f(s^{-1}t)$ for $f\in L^2(G, H_A)\cup L^2(G,H_B)$, then we define \[(\tilde{\pi}_A(a)f)(t)=\pi_A(\alpha_{t^{-1}}(a))f(t),\;\;\;(\tilde{\pi}_B(b)f)(t)=\pi_B(\beta_{t^{-1}}(b))f(t).\] We get a covariant representation $(\tilde{\pi}_A, \tilde{\tau}_{\Hh},\tilde{\pi}_B, \lambda, \lambda) $ of $\Hh$  by taking
\[(\tilde{\tau}_{\Hh}(\xi)f)(t)=\tau_{\Hh}(\rho_{t^{-1}}(\xi))f(t)\]
for $f\in L^2(G,H_B)$ and $(\tilde{\pi}_A\times \lambda, \tilde{\tau}_{\Hh}\times \lambda, \tilde{\pi}_B\times \lambda)$ is a representation of $\Hh\rtimes_\rho G$ into $\Ll(L^2(G, H_B), L^2(G,H_A))$. Its image $(\tilde{\tau}_{\Hh}\times \lambda)(\Hh\rtimes_\rho G)$ is called the reduced crossed product, denoted $\Hh\rtimes_{\rho, r}G$, which is a $C^*$-correspondence from $A\rtimes_{\alpha, r}G$ to $B\rtimes_{\beta, r}G$. If $G$ is amenable, then
$\Hh\rtimes_\rho G\cong \Hh\rtimes_{\rho, r}G$.
\end{example}

 \begin{rmk} The crossed product $\Hh\rtimes_\rho G$  can be characterized as universal for covariant representations.
  \end{rmk}

 Consider $\Cc$ the category in which the objects are $C^*$-algebras and in which the morphisms from $A$ to $B$ are the isomorphism classes of full $C^*$-correspondences from $A$ to $B$. The composition of $[\Hh]:A\to B$ with $[\Mm]:B\to C$ is the isomorphism class of $\Hh\otimes_B\Mm$ as a correspondence from $A$ to $C$. The isomorphisms in $\Cc$ are given by imprimitivity bimodules, see \cite{EKQR}.

Denote by $\Aa(G)$ the category in which the objects are $C^*$-algebras with actions of $G$ and in which the morphisms from $(A, G, \alpha)$ to $(B, G, \beta)$ are equivariant isomorphism classes of full $C^*$-correspondences from $A$ to $B$ with compatible actions  of $G$. The composition of $[\Hh, G, \rho]:(A, G, \alpha)\to (B, G,\beta)$ with $[\Mm, G, \sigma]:(B, G, \beta)\to (C, G, \gamma)$ is the isomorphism class of the tensor product action $(\Hh\otimes_B\Mm, G, \rho\otimes_B\sigma)$. 

\begin{rmk}
There is a functor  from the category $\Aa(G)$ to $\Cc$ taking a dynamical system $(A, G,\alpha)$ into $A\rtimes_\alpha G$ and a $C^*$-correspondence $(\Hh, G, \rho)$ into $\Hh\rtimes_\rho G$ (see section 3 in \cite{EKQR}).

In particular, for $G$ a compact group, $A\rtimes_\alpha G$ can be identified with a subalgebra of $A\otimes \Kk(L^2(G))$, see Proposition 4.3 in \cite{R} and $\Hh\rtimes_\rho G$ with a subspace of $\Hh\otimes \Kk(L^2(G))$.
\end{rmk}
\bigskip
\section{Group actions on Cuntz-Pimsner algebras and Morita-Rieffel equivalence}

\bigskip

If $A=B$, a  $C^*$-correspondence $\Hh$ from $A$ to $A$ is called a $C^*$-correspondence over $A$.
An action $\rho$ of $G$ on the $C^*$-correspondence ${\mathcal H}$ over $A$ determines an action on $\Kk(\Hh)={\mathcal K}_A({\mathcal H})$, the $C^*$-algebra generated by the finite rank operators $\theta_{\xi, \eta}$. The action is given  by $g\cdot\theta_{\xi,\eta}=\theta_{g\xi, g\eta}$ and it follows that  \[(g\cdot T)(\xi)=\rho_g(T(\rho_{g^{-1}}\xi))\] for $T\in \Kk(\Hh)$. Moreover,  there is an isomorphism \[\Kk(\Hh)\rtimes G\cong \Kk(\Hh\rtimes G),\] see \cite{HN}. Using the universal property, we get an action of $G$ on the Cuntz-Pimsner algebra ${\mathcal O}_A(\mathcal H)$ defined in \cite{K}. It is known that   $\Oo_A(\Hh)$ has a gauge action defined by $z\cdot a=a, \;\;z\cdot \xi=z\xi$ where $z\in \TT$. The action of $G$ commutes with the gauge action, therefore we  get  an action  on the core algebra ${\mathcal O}_A(\mathcal H)^{\mathbb T}$, the fixed point algebra under the gauge action. We recall the following result:

\begin{thm}\label{HN} (G. Hao, C.-K. Ng, \cite{HN}). Let ${\mathcal H}$ be a $C^*$-correspondence over $A$ and let the  locally compact amenable group $G$ act on $({\mathcal H},A)$ via $(\rho, \alpha)$.  Then $G$ acts on ${\mathcal O}_A(\Hh)$ via $\gamma$ and
\[{\mathcal O}_{A\rtimes_\alpha G}(\Hh\rtimes_\rho G)\cong {\mathcal O}_A(\Hh)\rtimes_\gamma G.\]
\end{thm}

\begin{cor} For $G$ locally compact amenable acting on  ${\mathcal H}$  we have
\[{\mathcal O}_{A\rtimes_\alpha G}(\mathcal H\rtimes_\rho G)^{\mathbb T}\cong {\mathcal O}_A(\mathcal H)^{\mathbb T}\rtimes_\gamma G.\]
\end{cor}

\begin{example}
Let $G$ be the symmetric group \[S_3=\la t,s\ra=\{1,t,ts, ts^2, s, s^2\},\]  which acts on $\Hh=\CC^2$ by \[t\cdot e_1=e_2, t\cdot e_2=e_1, s\cdot e_1=we_1, s\cdot e_2=w^2e_2.\] Here $t=(12), s=(123)$, $\{e_1,e_2\}$ is the canonical basis in $\CC^2$, and $w^2+w+1=0$. 

The Hilbert space $\Hh=\CC^2$ is viewed as a $C^*$-correspondence over $A=\CC$ with the usual operations. The group acts trivially on $A$, hence  \[A\rtimes S_3\cong C^*(S_3)\cong \CC\oplus \CC\oplus \MM_2.\] The second isomorphism is obtained by using the minimal central projections
\[p_1=\frac{1}{6}\chi_\iota,\; p_2=\frac{1}{6}\chi_\ve,\; p_3=\frac{1}{3}\chi_\sigma,\]
where 
\[\chi_\iota(1)=\chi_\iota(t)=\chi_\iota(s)=1;\;  \chi_\ve(t)=-1,\; \chi_\ve(1)=\chi_\ve(s)=1,\]\[ \chi_\sigma(1)=2,\;  \chi_\sigma(t)=0,\; \chi_\sigma(s)=-1\]
are the characters of the trivial representation $\iota$, the signature representation $\ve$ and the $2$-dimensional irreducible representation $\sigma$ defined by
\[\sigma(t)=\left[\begin{array}{rr}0&1\\1&0\end{array}\right],\; \sigma(s)=\left[\begin{array}{rr}w&0\\0&w^2\end{array}\right].\]
The crossed product $C^*$-correspondence $\Hh\rtimes S_3$ over $C^*(S_3)$ is a $12$-dimensional vector space with basis $\{\xi_x^j:x\in S_3, j=1,2\}$ where 
$\xi_x^j(y)=e_j$ for $y=x$ and $\xi_x^j(y)=0$ for $y\neq x$. 

To decompose  $\Hh\rtimes S_3$  into pieces according to the decomposition of $C^*(S_3)\cong \CC\oplus \CC\oplus \MM_2$, we need to compute the products $p_i\xi_x^jp_k$ for $i,k\in \{1,2,3\}, x\in S_3$ and $ j\in\{1,2\}$. Now
\[(p_i\xi_x^jp_k)(y)=\sum_{u\in S_3}\left(\sum_{v\in S_3}p_i(v)v\cdot \xi_x^j(v^{-1}u)\right)u\cdot p_k(u^{-1}y)=\]\[=\sum_{u\in S_3}p_i(ux^{-1})p_k(u^{-1}y)(ux^{-1})\cdot e_j.\]
A long computation  shows that
\[p_i(\Hh\rtimes S_3)p_k=0\;\text{for}\;  i,k\in \{1,2\},\]
\[\dim p_i(\Hh\rtimes S_3)p_3=\dim p_3(\Hh\rtimes S_3)p_k=2\;\text{for}\; i, k\in \{1,2\}\]and\[ \dim p_3(\Hh\rtimes S_3)p_3=4.\]
Since $\Oo_A(\Hh)\cong\Oo_2$, using Theorem \ref{HN} and results in \cite{D}, we get that $\Oo_2\rtimes S_3\cong \Oo_{C^*(S_3)}(\Hh\rtimes S_3)$ is Morita-Rieffel equivalent to the graph algebra with incidence matrix 
\[B=\left[\begin{array}{ccc}0&0&1\\0&0&1\\1&1&1\end{array}\right].\]
In particular, it follows that $K_0(\Oo_2\rtimes S_3)\cong \ZZ_2$ and $ K_1(\Oo_2\rtimes S_3)\cong0$.
Since the core algebra $\Oo_2^\TT$ is the UHF-algebra $M_{2^\infty}$, we get an action of $S_3$ on $M_{2^\infty}$ and $K_0(M_{2^\infty}\rtimes S_3)\cong\varinjlim(\ZZ^3,B)$.

\end{example}

\begin{dfn}
Given $C^*$-correspondences $\Hh$ and $\Mm$ over $A$ and $B$ respectively, we say that $\Hh$ and $\Mm$ are  Morita-Rieffel equivalent in case $A$ and $B$ are  Morita-Rieffel equivalent via an imprimitivity bimodule $\Zz$ such that $\Zz\otimes_B\Mm$ and $\Hh\otimes_A\Zz$ are isomorphic as  $C^*$-correspondences  from $A$ to $B$. 
\end{dfn}

Using linking algebras, Muhly and Solel proved in \cite{MS} that for faithful and essential Morita-Rieffel equivalent $C^*$-correspondences $\Hh$ and $\Mm$, the Cuntz-Pimsner algebras $\Oo_A(\Hh)$ and $\Oo_B(\Mm)$ are Morita-Rieffel equivalent. This result was generalized for possibly nonfaithful correspondences in \cite{EKK}.

\begin{rmk} A $C^*$-correspondence  $\Hh$ over $A$ determines an imprimitivity bimodule between $\Kk(\Hh)$  and $A$. The linking algebra $C=\Kk(\Hh\oplus A)$ consists of matrices
\[\left[\begin{array}{cc}T&\xi\\\tilde\eta&a\end{array}\right],\]
where $T\in \Kk(\Hh), \xi\in \Hh, \tilde\eta\in \tilde\Hh, a\in A$. Here $\tilde\Hh=\Kk(\Hh, A)$ denotes the dual of the Hilbert module $\Hh$. 

Assume now that the group $G$ acts on $\Hh$. Then $G$ acts on $\Hh\oplus A$ componentwise and $(\Hh\oplus A)\rtimes G\cong \Hh\rtimes G\oplus A\rtimes G$. If we write
\[C=\left[\begin{array}{cc}\Kk(\Hh)&\Hh\\\tilde{\Hh}&A\end{array}\right],\]
then $G$ acts on $C=\Kk(\Hh\oplus A)$  and 
\[C\rtimes G\cong \Kk(\Hh\rtimes G\oplus A\rtimes G)=\left[\begin{array}{cc}\Kk(\Hh\rtimes G)&\Hh\rtimes G\\\tilde{\Hh}\rtimes G&A\rtimes G\end{array}\right].\] 
In particular,  $\Hh\rtimes G$ determines an imprimitivity bimodule between $\Kk(\Hh\rtimes G)\cong \Kk(\Hh)\rtimes G$  and $A\rtimes G$. 

In order to interpret $\Hh\rtimes G$ as a $C^*$-correspondence  over $A\rtimes G$, we need to define the left multiplication using the map $A\rtimes G\to \Ll(\Hh\oplus A)\rtimes G\cong M(C)\rtimes G$ induced by $\phi: A\to \Ll(\Hh)$.
\end{rmk}

\begin{example}
Suppose that $\Hh$ is a $C^*$-correspondence over $A$ and $\ZZ_2$ acts via an order $2$ automorphism $\alpha$ of $A$ and an order $2$ automorphism $\sigma$ of $\Hh$. Recall that the crossed product $A\rtimes_\alpha \ZZ_2$ is isomorphic with the subalgebra of $\MM_2(A)$ made of matrices \[\left[\begin{array}{cc} a&b\\\alpha(b)&\alpha(a)\end{array}\right]\] with $a,b\in A$. Using the linking algebra, the crossed product $\Hh\rtimes_\sigma \ZZ_2$ is a $C^*$-correspondence over $A\rtimes_\alpha \ZZ_2$ and it can be understood as the set of $2\times 2$ matrices  \[\left[\begin{array}{cc} x&y\\\sigma(y)&\sigma(x)\end{array}\right]\] with $x,y\in \Hh$.
The operations are determined by the equations
\small{
\[\la \left[\begin{array}{cc} x&y\\\sigma(y)&\sigma(x)\end{array}\right], \left[\begin{array}{cc} x'&y'\\\sigma(y')&\sigma(x')\end{array}\right]\ra=\]\[=\left[\begin{array}{cc}\la x,x'\ra+\la y,\sigma(y')\ra&\la x,y'\ra +\la y, \sigma(x')\ra\\\la\sigma(y),x'\ra+\la \sigma(x),\sigma(y')\ra&\la\sigma(y),y'\ra+\la\sigma(x),\sigma(x')\ra\end{array}\right],\]

\bigskip

\[\left[\begin{array}{cc} x&y\\\sigma(y)&\sigma(x)\end{array}\right]\left[\begin{array}{cc} a&b\\\alpha(b)&\alpha(a)\end{array}\right]=\left[\begin{array}{cc} xa+y\alpha(b)&xb+y\alpha(a)\\\sigma(y)a+\sigma(x)\alpha(b)&\sigma(y)b+\sigma(x)\alpha(a)\end{array}\right],\]

\bigskip

\[\left[\begin{array}{cc} a&b\\\alpha(b)&\alpha(a)\end{array}\right]\left[\begin{array}{cc} x&y\\\sigma(y)&\sigma(x)\end{array}\right]=\left[\begin{array}{cc}ax+b\sigma(y)&ay+b\sigma(x)\\\alpha(b)x+\alpha(a)\sigma(y)&\alpha(b)y+\alpha(a)\sigma(x)\end{array}\right].\]
}
\end{example}

\bigskip

\begin{thm}\label{equi}
Suppose  that a localy compact amenable group $G$ acts on faithful and essential  Morita-Rieffel equivalent $C^*$-correspondeces $\Hh$ and $\Mm$ over $A$ and $B$ respectively, via an imprimitivity bimodule $\Zz$. Then $\Zz\rtimes G$ becomes an imprimitivity bimodule between $A\rtimes G$ and $B\rtimes G$. Moreover, $\Oo_A(\Hh)\rtimes G$ is Morita-Rieffel equivalent to $\Oo_B(\Mm)\rtimes G$.
\end{thm}
\begin{proof}
We use results of Combes, see \cite{C}, to realize $\Zz\rtimes G$ as an imprimitivity bimodule between $A\rtimes G$ and $B\rtimes G$. The isomorphism $\Zz\otimes_B\Mm\cong \Hh\otimes_A\Zz$ is equivariant, and \[(\Zz\otimes_B\Mm)\rtimes G\cong (\Zz\rtimes G)\otimes _{B\rtimes G}(\Mm\rtimes G)\cong\]\[\cong (\Hh\rtimes G)\otimes_{A\rtimes G}(\Zz\rtimes G)\cong (\Hh\otimes_A\Zz)\rtimes G,\]
hence $\Hh\rtimes G$  is Morita-Rieffel equivalent to $\Mm\rtimes G$ via $\Zz\rtimes G$. By the results in \cite{MS} combined with \cite{HN}, it follows that $\Oo_A(\Hh)\rtimes G$ is Morita-Rieffel equivalent to $\Oo_B(\Mm)\rtimes G$.
\end{proof}

\begin{rmk}
Note that for $G$ compact, the set of fixed points \[\Hh^G=\{\xi\in \Hh\;\mid\; g\xi=\xi\}\] becomes a $C^*$-correspondence over the fixed point algebra $A^G$. If the action on $A$ is saturated, then  $A^G$ is Morita-Rieffel equivalent with $A\rtimes G$ and   it follows that $\Hh^G$ is Morita-Rieffel equivalent with $\Hh\rtimes G$. In particular, $\Oo_A(\Hh)\rtimes G$ is Morita-Rieffel equivalent with $\Oo_{A^G}(\Hh^G)$, which in some cases might be easier to handle.
\end{rmk}
\begin{rmk}
Consider $j_\Hh:\Hh\to \Oo_A(\Hh)$ the canonical inclusion and let $G$ be a locally compact abelian group acting on $\Hh$. Then the dual $\hat{G}$ acts on $A\rtimes G$ and $\Hh\rtimes G$ in the usual way. Using the map $j_\Hh$ and Takai duality, we obtain   \[(\Hh\rtimes G)\rtimes \hat{G}\subset \Oo_{(A\rtimes G)\rtimes \hat{G}}((\Hh\rtimes G)\rtimes \hat{G})\cong\]\[\cong(\Oo_A(\Hh)\rtimes G)\rtimes \hat{G}\cong \Oo_A(\Hh)\otimes \Kk(L^2(G)).\]
A  duality result for crossed products of Hilbert $C^*$-modules by group  coactions was obtained by Kusuda, see \cite{Ku}.
\end{rmk}

\bigskip

\section{Applications to Cuntz-Pimsner algebras of vector bundles}

\bigskip

Consider $E\to X$  a Hermitian vector bundle of rank $n$, where $X$ is compact, metrizable and path connected. The set of continuous sections $\Gamma(E)$ becomes a $C^*$-correspondence over $C(X)$, with left and right multiplications  given by \[(f\xi)(x)=(\xi f)(x)=f(x)\xi(x)\]
for $f\in C(X), \xi\in \Gamma(E)$ and inner product \[\la \xi, \eta\ra(x)=\la \xi(x), \eta(x)\ra_{E_x}.\] We denote by $\Oo_E$ the Cuntz-Pimsner algebra of the $C^*$-correspondence $\Gamma(E)$ over $C(X)$, which is a (locally trivial) continuous field of Cuntz algebras $\Oo_n$, see Proposition 2 in \cite{V}. In particular, $\Oo_E$ is a $C(X)$-algebra.

\begin{rmk}
The $C^*$-algebra $\Oo_E$ is generated by $C(X)$ and $S_1,...,S_N$ such that
\[f \;S_j=S_j \;f,\; S_j^*S_k=P_{jk},\; \sum_{j=1}^NS_jS_j^*=1\]
where $f\in C(X)$ and $P\in M_N\otimes C(X)$ is the projection that gives the bundle $E$ by the Serre-Swan Theorem.
Since $\Gamma(E)$ is a direct summand of $C(X)^N$, the Hilbert module $\Gamma(E)\rtimes G$ is finitely generated and projective since it is a direct summand of  $C(X)^N\rtimes G\cong (C(X)\rtimes G)^N$.
\end{rmk}

Recall that the cohomology set $H^1(X,U_n)$ classifies the Hermitian vector bundles of rank $n$ over $X$, where $U_n=U_n(\CC)$ denotes the unitary group. The set of isomorphism classes of locally trivial continuous fields over $X$ of $C^*$-algebras with fiber $A$  is in bijection with the cohomology set $H^1(X$, Aut$(A))$. If $(X_i,\alpha_{ij})$ is the cocycle corresponding to the bundle $E$, then the cocycle with values in Aut$(\Oo_n)$ describing $\Oo_E$ is given by 
\[\alpha_{ij}^*(x)(t)=\alpha_{ij}(x)^{\otimes (r+k)}t(\alpha_{ij}(x)^{\otimes r})^*,\;\text{where}\; t\in (n^r,n^{r+k})\subset \Oo_n.\]
Here $\Oo_n$ is viewed as the Doplicher-Roberts algebra associated to the $n$-dimensional Hilbert space $H$, and $(n^r, n^{r+k})$ denotes the set of linear maps $H^{\otimes(r+k)}\to H^{\otimes r}$. We reproduce the following example from \cite{V2}, which shows that for some vector bundles, the field of Cuntz algebras associated to $\Oo_E$ is not trivial.

\begin{example}
Let $X=S^{2k}$ be an even sphere. Recall that $H^1(S^{2k},U_n)\cong \pi_{2k-1}(U_n)$ for $k\ge 1$ and that \[K^0(S^{2k})\cong \ZZ^2, \; K^1(S^{2k})\cong 0.\]  Moreover,  as a ring, $K^0(S^{2k})$ is isomorphic to  $\ZZ[\lambda]/(\lambda^2)$. Tensoring with a fixed vector bundle $E\to S^{2k}$ with class $n+m\lambda$ defines an endomorphism $[E]$ of $K^0(S^{2k})$ given by 
\[[E](x+y\lambda)=nx+(ny+mx)\lambda.\]
Using the six-term exact sequence for the $K$-theory of $\Oo_E$, see \cite{K}  we get
\[\begin{array}{ccccc}\ZZ^2&\stackrel{id-[E]}\longrightarrow&\ZZ^2 &\stackrel{i_0}\longrightarrow& K_0(\Oo_E)\\
0\bigg\uparrow\hspace{2mm} &&{}&&\hspace{2mm}\bigg\downarrow 0\\
K_1(\Oo_E)&\stackrel{i_1}\longleftarrow& 0 & \stackrel{0}\longleftarrow&0\end{array}\]
hence
\[K_0(\Oo_E)\cong\ZZ^2/(id-[E])\ZZ^2,\; K_1(\Oo_E)=0.\]
In the case the bundle $E$ is such that $n-1$ and $m$ are relatively prime and $n\ge 3$, we get \[K_0(\Oo_E)\cong \ZZ/(n-1)^2\ZZ,\] hence  $\Oo_E$ determines a non-trivial field of Cuntz algebras since by the Kunneth formula we have
\[K_0(C(S^{2k})\otimes \Oo_n)\cong \ZZ/(n-1)\ZZ\oplus \ZZ/(n-1)\ZZ.\]
\end{example}
\begin{rmk}
For $X=S^{2k+1}$ an odd sphere and for $[E]=n\in K^0(S^{2k+1})\cong \ZZ$,  Vasselli showed in \cite{V} that we have \[(\Oo_E\otimes\Kk,\ZZ)\cong (C(S^{2k+1})\otimes \Oo_n\otimes\Kk,\ZZ)\]
as graded algebras.
\end{rmk}

\begin{rmk}
If $E$ is a line bundle over $X$, then $\Oo_E$ is commutative with spectrum homeomorphic to the circle bundle of $E$. If $E, F$ are line bundles over $X$, then $\Oo_E \cong \Oo_F$ as $C(X)$-algebras if and only if $E\cong F$ or $E \cong  \bar{F}$, where $\bar{F}$ is the conjugate of $F$, see \cite{Da2}, Proposition 2.2. 
\end{rmk}
In general, for a Hermitian vector bundle $E\to X$ of rank $n$, where $X$ is compact, metrizable and path connected, Dadarlat showed in \cite{Da2} that the principal ideal $(1- [E ])K^0 (X )$ of the $K$-theory ring $K^0(X)$ determines $\Oo_E$ up to isomorphism and an inclusion of principal ideals $(1 - [E])K^0(X)\subseteq (1 - [F])K^0(X)$ corresponds to unital embeddings $\Oo_E \subseteq \Oo_F$ . In particular if  $n\ge 2$, then 
$\Oo_E \cong C(X)\otimes  \Oo_{n}$ if and only if $[E] - 1$ is divisible by $n-1$,  see \cite {Da2}, Theorem 1.1.

If the compact group $G$ acts on the Hermitian vector bundle $E\to X$ by isometries, we define
\[(g\cdot f)(x)=f(g^{-1}x)\; \text{for}\; f\in C(X),\;\;(g\cdot\xi)(x)=g\xi(g^{-1}x)\;\text{for}\; \xi\in \Gamma(E).\]
A section $\xi\in \Gamma(E)$ is invariant if $g\cdot\xi=\xi$. The set of invariant sections forms a subspace $\Gamma(E)^G$ of $\Gamma(E)$. The averaging operator defines a map $\Gamma(E)\to\Gamma(E)^G$.

 Using the universal property,  $G$ acts on the Cuntz-Pimsner algebra $\Oo_E$ and from Theorem \ref{HN}, we get \[\Oo_E\rtimes G\cong \Oo_{C(X)\rtimes G}(\Gamma(E)\rtimes G).\]

 The structure of the crossed product $C(X)\rtimes G$ is known under certain hypotheses. For example, Williams showed in \cite{W} that if $X/G$ is Hausdorff, then $C(X)\rtimes G$ is a section algebra over $X/G$ with fiber over the orbit $Gx$ isomorphic to $C^*(G_x)\otimes\Kk(L^2(G/G_x),\mu)$.  Here $G_x$ is the stabilizer group  and $\mu$ is  any quasi-invariant measure on $G/G_x$.  
In some cases, the study of $C^*$-correspondences over $C(X)\rtimes G$ can be reduced  to understanding the $C^*$-correspondences over abelian $C^*$-algebras $C_0(Y)$ or over continuous trace algebras. 

 \begin{rmk}
 An imprimitivity bimodule over a commutative $C^*$-algebra $C_0(Y)$ can  be regarded as the space of sections of a complex line bundle over $Y$, see Appendix (A) in \cite{Ra}. A general $C^*$-correspondence $\Hh$ over an abelian $C^*$-algebra $C_0(Y)$ is given by the  sections of a continuous field of Hilbert spaces $\{\Hh_y\}$ over $Y$ and by a continuous family of measures $\{\mu_y\}$ on $Y$ which determine the left multiplication $\phi:C_0(Y)\to \Ll(\Hh)$ via a family of  representations $C_0(Y)\to \Ll(\Hh_y)$. If all fibers $\Hh_y$ are nonzero, then $\Hh$ is full. The Cuntz-Pimsner algebras $\Oo_{C_0(Y)}(\Hh)$ recover  large classes of $C^*$-algebras.
 \end{rmk}
 \begin{examples}
 Particular cases of $C^*$-correspondences $\Hh$ over an abelian $C^*$-algebra arise from  topological quivers $(E^0, E^1,s,r,\lambda)$, see  \cite{MT},  or from topological relations, see \cite{Br}. In the first case,   $\Hh$ is obtained as a completion of  $C_c(E^1)$ and the fibers are $\Hh_v=L^2(s^{-1}(v),\lambda_v)$ for $v\in E^0$. Here $\lambda=\{\lambda_v\}_{v\in E^0}$ is a family of Radon measures which defines the inner products with values in $C_0(E^0)$, the source map $s:E^1\to E^0$  defines the right multiplication, $\xi\cdot f(e)=\xi(e)f(s(e))$  and the range map $r:E^1\to E^0$  defines the left multiplication $\phi:C_0(E^0)\to \Ll(\Hh)$, $\phi(f)\xi(e)=f(r(e))\xi(e)$,  for $e\in E^1, \xi\in C_c(E^1), f\in C_0(E^0)$.  
 
This $C^*$-correspondence $\Hh$ over $A=C_0(E^0)$ is essential.  The left multiplication $\phi:C_0(E^0)\to \Ll(\Hh)$ is injective when $r$ is onto, and $\phi(f)\in \Kk(\Hh)$  if and only if  $f\circ r\in C_0(E^1)$ and for every $e\in E^1$ such that $(f\circ r)(e)\neq 0$ there exists a neighborhood $U$ of $e$ such that the restriction $s\mid_U:U\to s(U)$ is a homeomorphism.

If $s: E^1\to E^0$ is a local homeomorphism and $\lambda_v$ are the counting measures, we recover the situation of $C^*$-correspondences associated to toplogical graphs of Katsura, see  \cite{K1}.  
\end{examples}
\begin{example}
Another source of examples of $C^*$-correspondences over  abelian $C^*$-algebras comes from the concepts of diagonal and of  Cartan subalgebra $B$ in a  $C^*$-algebra $A$, see \cite{Kum} and \cite{Re}. Recall that in particular $B$ is a maximal abelian subalgebra  in $A$ and there is a faithful conditional expectation $P:A\to B$, so one can define a $B$-valued  inner product on $A$ by $\la a_1, a_2\ra=P(a_1^*a_2)$.  By taking a completion and using the obvious left and right multiplications, one obtains a $C^*$-correspondence $\Hh$ over $B=C_0(\hat{B})$. It would be interesting to explore what kind of Cuntz-Pimsner algebras $\Oo_B(\Hh)$ one obtain in this way. For example, if $A=\MM_n$ and $B\cong \CC^n$ is the subalgebra of diagonal matrices, then we get the Cuntz algebra $\Oo_n$.
\end{example}


To explore the structure of  the crossed product $\Oo_E\rtimes G$, we first consider  the case of trivial bundles. Let $E=X\times V$, where $V\cong \CC^n$. Then $\Gamma(E)\cong C(X)\otimes V$ and  $\Oo_E\cong C(X) \otimes \Oo_n$. A group action on $X\times V$ is of the form
$g(x,v)=(gx,J(g,x)v)$, where $J:G\times X\to End(V)$ is continuous. We must have $J(e,x)=I$ for all $x\in X$, and from associativity of the action, we obtain the cocycle relation $J(gh,x)=J(g, hx)J(h,x)$.
The structure of $\Oo_E\rtimes G\cong (C(X)\otimes \Oo_n)\rtimes G$ can be determined in principle by understanding the $C^*$-correspondence $(C(X)\otimes V)\rtimes G$ over $C(X)\rtimes G$.

In particular, if the action of $G$ on $E=X\times V$ is free, then $E/G\cong (X/G)\times V$ and  $\Gamma(E/G)\cong C(X/G)\otimes V$. Since $C(X)\rtimes G$ is Morita-Rieffel equivalent with $C(X/G)$ and $(C(X)\otimes V)\rtimes G$ is Morita-Rieffel equivalent with $C(X/G)\otimes V$, it follows that $\Oo_E\rtimes G$ is Morita-Rieffel equivalent with $C(X/G)\otimes \Oo_n$.

If $G$ acts trivially on $X$, then $C(X)\rtimes G\cong C(X)\otimes C^*(G)$ and $(C(X)\otimes V)\rtimes G\cong C(X)\otimes (V\rtimes G)$, so $\Oo_E\rtimes G\cong C(X)\otimes (\Oo_n\rtimes G)$.

 For the more general structure of  $\Oo_E\rtimes G$  we have the following results.

\begin{thm} (Free action). If $G$ compact acts freely on the Hermitian vector bundle $E\to X$, then $\Oo_E\rtimes G$ is Morita-Rieffel equivalent with a continuous field of Cuntz algebras over $X/G$. 
\end{thm}
\begin{proof} Indeed, in this case $C(X)\rtimes G$ is Morita-Rieffel equivalent with $C(X/G)$ via the imprimitivity bimodule $\Zz$ which is a completion of $C(X)$, and $\Gamma(E)\rtimes G$ is Morita-Rieffel equivalent with $\Gamma(E)^G\cong \Gamma(E/G)$.
\end{proof}

 \begin{example}
 Let $X=S^2$ be the $2$-sphere and let $E=TS^2\otimes \CC$ be the complexified tangent bundle. Even though $E$ is not a trivial bundle, we have $\Oo_E\cong C(S^2)\otimes \Oo_2$, see Theorem 1.1 in \cite{Da1}. 
 
 The group $\ZZ_2=\{e,g\}$ acts on $S^2$ by $g\cdot x=-x$ and on $E$ by its differential $dg$. Since the action is free, $E/\ZZ_2$ is a vector bundle over $S^2/\ZZ_2=\RR P^2$. Moreover,  $C(S^2)\rtimes \ZZ_2$ is Morita-Rieffel equivalent with $C(\RR P^2)$ and  it follows that $\Oo_E\rtimes \ZZ_2$ is Morita-Rieffel equivalent with $C(\RR P^2)\otimes \Oo_2$.
 \end{example}

\begin{thm} (Fiberwise action). If $G$ compact acts  on the Hermitian vector bundle $E\to X$ of rank $n$ and the action on $X$ is trivial, then $\Oo_E\rtimes G$ is a continuous field over $X$ with fibers $\Oo_n\rtimes G$.
\end{thm}
\begin{proof} Indeed, in this case we have a fiberwise action on the continuous field of Cuntz algebras $\Oo_n$ and we can use Theorem 4.1 in \cite{KW}.
\end{proof}

\begin{rmk}
For $G=S_n$ the symmetric group, we know from \cite{D} that $\Oo_n\rtimes S_n$ is simple and purely infinite.
If $X$ is finite dimensional, Dadarlat gives in \cite{Da1} a complete list of the UCT Kirchberg algebras $D$ with finitely generated $K$-theory  for which every unital separable continuous field over $X$  with fibers isomorphic to $D$ is automatically locally trivial or trivial. 
\end{rmk}

\begin{thm} (Transitive action). Let $G$ be a compact group  and let $H$ be a closed subgroup. Given a Hermitian vector bundle $E$ over $X=G/H$ we know that $E\cong G\times_HV$ for an $H$-module $V$. Then $\Oo_E\rtimes G$ is Morita-Rieffel equivalent to a graph $C^*$-algebra. 
\end{thm}
\begin{proof}
Indeed, $C(G/H)\rtimes G$ is Morita-Rieffel equivalent with $C^*(H)$ which is a direct sum of matrix algebras. This in turn is Morita-Rieffel equivalent to $C_0(Y)$ with $Y$ at most countable. Now it is known that
a $C^*$-correspondence over $C_0(Y)$  gives rise to a discrete graph. The result follows from Theorem  \ref{equi}  since $\Oo_E\rtimes G\cong \Oo_{C(G/H)\rtimes G}(\Gamma(E)\rtimes G)$.

\end{proof}

\end{document}